\documentclass[12pt]{amsart}

\usepackage{amsmath,mathrsfs}
\usepackage{amsthm}
\usepackage{amssymb}
\usepackage{enumerate}
\usepackage{xcolor}
\usepackage{comment}
\usepackage{tikz}
\usepackage{tikz-cd}
\usepackage{color} 
\definecolor{rossred}{rgb}{1.0,0.25,0.66}  
\definecolor{rossgreen}{rgb}{0.25,0.66,0.25} 
\definecolor{rossblue}{rgb}{0.25,0.66,1.0}
\definecolor{sashapurple}{rgb}{0.5,0.15,0.5}
\usepackage[pdfborder={0 0 0}]{hyperref}
\hypersetup{
    colorlinks=true,
    linkcolor=rossblue,
    citecolor=rossgreen
}

\numberwithin{equation}{section}
\theoremstyle{plain}
\newtheorem{theorem}{Theorem}[section]
\newtheorem{proposition}[theorem]{Proposition}
\newtheorem{lemma}[theorem]{Lemma}

\newtheorem{corollary}[theorem]{Corollary}

\newtheorem{question}[theorem]{Question}
\theoremstyle{definition}

\newtheorem{definition}[theorem]{Definition}

\newtheorem{remark}[theorem]{Remark}

\newtheorem{example}[theorem]{Example}
\AtEndEnvironment{example}{\popQED\endexample}

\usepackage{url}
\usepackage{enumitem}
\usepackage[utf8]{inputenc}
\usepackage[T1]{fontenc}
\usepackage[paperwidth = 8.5in, paperheight = 11in, inner = 1in, outer = 1in, top = 1in, bottom = 1in]{geometry}

\newcommand{\trdeg}{\mbox{\rm trdeg}\,}

\newcommand{\cI}{\mathcal{I}}

\newcommand{\Z}{\mathbb Z}

\newcommand{\depth}{\text{depth}}

\newcommand{\im}{\text{im}}

\newcommand{\ara}{\text{ara}}

\definecolor{campsc}{rgb}{0.44, 0.0, 1.0}

\title{Analytic spread of binomial edge ideals}

\author[E. Camps-Moreno, et. al]{Eduardo Camps-Moreno }
\address[E. Camps-Moreno]
{Department of Mathematics, Virginia Tech, 225 Stanger Street
Blacksburg, VA 24061-1026}
\email{eduardoc@vt.edu}

\author[]{Deblina Dey}
\address[D. Dey]
{Department of Mathematics, Indian Institute of Technology Madras, Chennai, Tamil Nadu 600036, India}
\email{deblina.math@gmail.com}

\author[]{Souvik Dey}
\address[S. Dey]
{Department of Mathematical Sciences
850 West Dickson Street, University of Arkansas
Fayetteville, Arkansas 72701}
\email{souvikd@uark.edu}

\author[]{T\`ai Huy H\`a}
\address[T.H. H\`a]
{Department of Mathematics, Tulane University, 6823 St. Charles Avenue, New Orleans, LA 70118, USA}
\email{tha@tulane.edu}

\author[]{Stephen Landsittel}
\address[S. Landsittel]
{Institute of Mathematics, Hebrew University, Givat Ram, Jerusalem 91904, Israel}
\email{stephen.landsittel@mail.huji.ac.il}

\author[]{Benjamin Oltsik}
\address[B. Olstik]
{Department of Mathematics, University of Connecticut,
341 Mansfield Road, U1009,
Storrs, CT 06269-1009}
\email{broltsik@gmail.com}

\author[]{Shahriyar Roshan Zamir}
\address[S. Roshan Zamir]
{Department of Mathematics, Tulane University, 6823 St. Charles Avenue, New Orleans, LA 70118, USA}
\email{sroshanzamir@tulane.edu}

\author[]{Adam Van Tuyl}
\address[A. Van Tuyl]
{Department of Mathematics and Statistics
McMaster University, Hamilton, ON L8S 4L8, Canada}
\email{vantuyla@mcmaster.ca}

\date{\today}

\begin{document}
\begin{abstract}
We investigate the analytic spread $\ell(J_G)$ 
when $J_G$ is the binomial edge ideal 
of a finite simple graph $G$. For any connected
graph $G$ on $n$ vertices, we show that the bounds 
$n-1 \leq \ell(J_G) \leq 2n-3$ hold, and moreover, 
these bounds are tight. For some special families of graphs 
(e.g., closed graphs, pseudo-forests) we compute the exact value of the analytic spread of the corresponding
binomial edge ideal via combinatorial and convex geometric means. 
\end{abstract}

\maketitle

\section{Introduction}\label{sec-1}

The goal of this paper is to study the analytic spread of 
binomial edge ideals. Given a finite simple 
graph $G = (V,E)$ on $n$ vertices, 
its {\it binomial edge ideal} is defined as  
\[
J_G = (x_iy_j - x_jy_i \mid \{i,j\} \in E) 
\subset K[x_1,\ldots,x_n,y_1,\ldots,y_n],
\]
where $K$ is an infinite field. 
Binomial edge ideals were 
introduced independently in \cite{HHHKR} and \cite{ohtani2011graphs},
where foundational properties such as the primary decomposition,
minimal primes, and the Gr\"obner bases of $J_G$ were described. 
Binomial edge ideals continue to attract significant 
attention due to the rich interplay between their algebraic properties and the combinatorics of the underlying graph. 
Some examples of this recent work include
\cite{bolognini2022cohen, ene2022powers, kats, kumar2022rees,PeevaClosed,WT}.

The {\it analytic spread} of an ideal $I$, denoted $\ell(I)$, 
is an important invariant
in commutative algebra that measures the dimension of the special fiber ring
of $I$ (see Section \ref{sec-2} for the formal definition). 
The special fiber ring is the coordinate ring of the fiber over 
the origin of the blowup at the variety defined by $I$.
The study of the Rees algebra, first introduced in 
\cite{ReeseNortcotOG}, and the special fiber ring  have generated a tremendous 
amount of research. A comprehensive list of all such results is beyond our scope but some related papers are 
\cite{PoliniUlrichCarsoGhezzi,cutkosky2024reesalgebrasreducedfiber,HeinzerKim,
herzog2019fiberconemonomialideals, 
lin2021, 
saloni2022};
for an introduction, see the book \cite{SH}.
The value of $\ell(I)$ is also related to the minimal number of 
generators in a reduction of $I$ 
(e.g., \cite[Corollary 8.2.5]{SH}) and to the growth of the number of minimal
generators of $I^d$ as $d \gg 0$ (e.g., \cite{Burch_1972}).

While the analytic spread has been well-studied for monomial ideals 
(see \cite{bivia2003analytic,HN1}), this invariant is largely unexplored in the context of binomial (edge) ideals. Only a handful of results are known. Ene, Rinaldo, and Terai \cite{ene2022powers} showed that for \emph{closed graphs}, a subclass of chordal
graphs, the analytic spread of their binomial edge ideals coincides with that of their initial ideals. Kumar \cite{kumar2022rees} computed the analytic spread of the binomial edge ideal of closed graphs via the number of connected and indecomposable components. Beyond these special cases, the analytic spread of binomial edge ideals remains mysterious.

Our work takes the first steps towards better understanding the 
analytic spread of $J_G$ for general graphs $G$.   One of our main
results provides tight upper and lower bounds for the analytic spread of 
$\ell(J_G)$ for any graph $G$.  Specifically, we prove:

\begin{theorem}[Theorem~\ref{bounds}]
If $G$ be a connected graph on $n$ vertices, then  \begin{equation}\label{eq: intro bounds}    
n - 1 \leq \ell(J_G) \leq 2n - 3.
\end{equation}
\end{theorem}

\noindent
The proof of Theorem \ref{bounds} relies on a number of 
ingredients.  We require the general inequalities
${\rm ara}(I) \leq \ell(I) \leq \mu(I)$,
where ${\rm ara}(I)$ denotes the arithmetic rank of $I$, and
$\mu(I)$ is the number of minimal generators of $I$.   
We also make use of  the inequality 
$\ell(J_{G'}) \leq \ell(J_G)$ if $G'$ is subgraph
of $G$, which is proved in Theorem \ref{thm-add}.  Finally,
we also rely on an analysis of the special fiber ring of $J_G$ via classical techniques of transcendence bases to show that the upper bound in \eqref{eq: intro bounds} is achieved when $G$ is a complete graph. 
The lower bound is achieved when $G$ is a tree.  

Because each generator of $J_G$ corresponds to an edge of $G$, we have $\mu(J_G) = |E|$. Since $\ell(J_G) \leq \mu(J_G)$ always holds, 
it is natural to ask for what graphs do we have the equality? Section~\ref{sec-exact} investigates this question;  
our results in this direction include:

\begin{theorem}[Theorem~\ref{thm-main}]
Let $G$ be a pseudo-forest. Then $\ell(J_G) = |E|$. In particular, this holds when $G$ is a forest or a unicyclic graph.
\end{theorem}
\noindent
The main novelty of the proof of Theorem~\ref{thm-main} is the reduction to the case of connected graphs and the use of an additivity formula
that we prove in Theorem \ref{thm-components}:
\[
\ell(J_G) = \sum_{i=1}^r \ell(J_{G_i}),
\]  
where $G_1, \ldots, G_r$ are the connected components of $G$.

For closed graphs, we further present a geometric perspective by 
relating $\ell(J_G)$ to the \emph{Newton--Okounkov region} associated 
to $J_G$, thus contributing to the broader program
(see, for example \cite{Cut1,KK,LM,CL})
studying these regions.
Our Theorem \ref{thm-closed1} shows that 
$\ell(J_G)$, when $G$ is closed, equals the maximal dimension of a compact face of this convex region, drawing connections between algebraic invariants and convex geometry. Finally, in Corollary~\ref{cor: Pluck and bound} we recover a result of Kumar \cite{kumar2022rees} that a closed graph $G$ satisfies $\ell(J_G) = |E| = \mu(J_G)$ if and only if $G$ is $K_4$-free.

\noindent\textbf{Structure of the paper.} 
Section~\ref{sec-2} reviews the definitions of binomial edge ideals 
and analytic spread, along with relevant tools from transcendence base theory. 
Section~\ref{sec-3} develops general properties and bounds on $\ell(J_G)$. 
In Section~\ref{sec-4}, we prove Theorem~\ref{bounds}.
Section~\ref{sec-exact} focuses on exact computations for $\ell(J_G)$ from some 
families of graphs. In Section~\ref{sec-6}, we restrict to the case that
$G$ is a closed graph.  We present a connection between $\ell(J_G)$ 
and the Newton--Okounkov region of $J_G$ and classify the closed graphs
with $\ell(J_G) = \mu(J_G)$.

\section{Preliminaries}\label{sec-2}

In this section we fix notation and recall basic terminology.
Throughout this paper, let $K$ be an 
infinite field, let $R = K[x_1,\ldots,x_n]$ 
be the polynomial ring in $n$ variables, and let 
$S = K[x_1,\ldots,x_n,y_1,\ldots,y_n]$ be the polynomial ring 
in $2n$ variables.  Note that while some of our results
do not require $K$ to be infinite, some of our 
key tools, e.g., Remark \ref{rmk-ara}, require
this hypothesis, so we have elected to make
this global assumption on $K$.

All the graphs considered in this paper will be finite and simple. 
If $G = (V,E)$ is a graph on $n$ vertices, we will denote the vertices 
by the positive integers $1,\ldots,n$, as in $V = [n]:= \{1,\ldots,n\}$. For $n\geq 1$, the \textit{cycle} $C_n$ is the graph $C_n = ([n],E)$ whose edges are $E = \{\{1,2\},\{2,3\},\ldots,\{n-1,n\},\{n,1\}\}$. For a graph $G = (V,E)$, a \textit{subgraph} of $G$ is a graph $H = (V_1,E_1)$ such that $V_1\subset V$ and $E_1\subset E$. A subgraph $H = (V_1,E_1)$ of $G = (V,E)$ is called \textit{induced} if $E_1 = E\cap (V_1\times V_1)$. A graph $G$ is called a \textit{tree} if it does not contain any cycles. The binomial ideal
\begin{equation*}
    (x_iy_j-x_jy_i\mid \{i,j\}\in E)\subset S
\end{equation*}
is called the {\it binomial edge ideal} of the graph $G$, 
and is denoted by $J_G$.

In this paper we compute analytic spread of binomial edge ideals. The definition of the analytic spread of an ideal in a polynomial ring is as follows.

\begin{definition}\label{Def: Rees Alg/fib cone}
     Let $\mathfrak{m} = (x_1,\ldots,x_n)\subset R$, and let $I\subset R$ be an ideal. We define the \textit{Rees algebra} of $I$ to be
    \begin{equation*}
        R[It] = \bigoplus_{n\geq 0}I^nt^n.
    \end{equation*}
   The \textit{special fiber ring} of $I$ is $F(I) = R[It]/\mathfrak{m}R[It]$. The \textit{analytic spread} $\ell(I)$ of $I$ is defined to be $\dim F(I)$.
\end{definition}

  The next lemma about the analytic spread will be utilized in Sections \ref{sec-3} and \ref{sec-4}.

   \begin{lemma}[{\cite[Proposition 10.3.2]{MR2724673}}]\label{lem-spread}
       If $I\subset R$ is a graded ideal, then 
           \begin{equation*}
               \lim_{k\to\infty} {\rm depth} (R/I^k)
               \leq \dim(R) -\ell(I).
           \end{equation*}
   \end{lemma}

   In Section \ref{sec-3} we will prove that if $I\subset R$ is an ideal which is generated by forms of the same degree, then it follows from Lemma \ref{intro-spread} that $\ell(I)\leq \mu(I)$, the
   number of minimal generators of $I$.
   In particular, $\ell(J_G)\leq \mu(J_G) = |E|$ for any graph $G$, thus
   providing the upper bound for any graph.  A lower
   bound for $\ell(I)$ comes from the arithmetical rank.

\begin{definition}
    For an ideal $I$ in $R$, the {\it arithmetical rank of $I$}, 
    denoted $\operatorname{ara}(I)$, is defined as
    \begin{equation*}
        \min \{k\mid \text{ there exists } f_1,\ldots,f_k \text{ where } 
        \sqrt{I}=\sqrt{(f_1,\ldots,f_k}\},
    \end{equation*}
    where $\sqrt{I}$ is the radical of $I$. The arithmetical rank of 
    an ideal is the minimal number of the generators of $I$ up to radical.
\end{definition}

In Section \ref{sec-4} we make use of the following lower bound on the 
analytic spread. 

\begin{remark}\label{rmk-ara}
For any ideal $J$ in $R$,  the inequality
    \begin{equation}
        \operatorname{ara}(J) \leq \ell(J).
    \end{equation}
holds. This inequality for a local ring $R$ follows from 
\cite[Prop. 8.3.8]{SH}, but the proof extends naturally to the graded case. 
\end{remark}

In order to compute the analytic spread of binomial edge ideals in this paper, 
we compute the dimension of the special fiber ring of the ideal directly,
using transcendence base theory. We proceed by recalling some 
basic theory of transcendence bases.

\begin{definition}
    Suppose that $K\subset L$ is a field extension. If $E\subset L$, 
    then we say that the set $E$ is {\it algebraically independent} 
    over $K$ if the following holds:
    \begin{equation*}
        \begin{split}
            &\text{    for $r\geq 1$, $f_1,\ldots,f_r\in E$ and polynomials $G(X_1,\ldots,X_r)$ over $K$ such that}
            \\&\text{$G(f_1,\ldots,f_r)=0$, we have that $G(X_1,\ldots,X_r)$ is the zero polynomial. }
        \end{split}
    \end{equation*}
There exists a maximal set $B\subset L$ with respect to containment, which is algebraically independent over $K$, and all such sets have the same cardinality. A maximal set $B\subset L$ which is algebraically independent over $K$ is called a \textit{transcendence base} of $L$ over $K$. The \textit{transcendence degree} of $L$ over $K$, denoted $\trdeg_K L$, is defined to be the cardinality of any transcendence base of $L$ over $K$.
\end{definition}

We will make repeated use of the following important fact in Section~\ref{sec-3}.

\begin{theorem}[{\cite[Chapter 8, Theorem 1.1]{Lang}}]\label{thm-transbase}
    Let $K\subset L$ be fields. Any two transcendence bases of $L$ over $K$ have the same cardinality. If $\Gamma$ is a subset of $L$ such that $L$ is algebraic over $K(\Gamma)$, and $S$ is a subset of $\Gamma$ which is algebraically independent over $K$, then there exists a transcendence base $\mathscr{B}$ of $L$ over $K$ such that $S\subset \mathscr{B}\subset \Gamma$.
\end{theorem}

The following fact from commutative algebra will also be useful for computing analytic spread in Section~\ref{sec-3} and Section~\ref{sec-4}.

\begin{remark}\label{rmk-trdeg}
    Let $A$ be a domain and let $\text{QF}(A)$ be the quotient field of $A$. If $A$ is a finitely generated algebra over $K$, then
\begin{equation*}
    \dim(A) = \trdeg_K \text{QF}(A).
\end{equation*}
\end{remark}

We will use Remark \ref{rmk-trdeg} combined with 
Theorem \ref{thm-kit} given below 
in order to compute analytic spread of binomial edge ideals by directly calculating the transcendence degree of their special fiber rings.

\section{General properties of the analytic spread of a 
binomial edge ideal}\label{sec-3}

In this section, we establish two general results about the analytic spread of 
binomial edge ideals. Theorem~\ref{thm-components} allows us to reduce to the 
case of connected graphs, while Theorem~\ref{thm-add} compares the analytic 
spread of the binomial edge ideal of a graph to that of a subgraph with the 
same number of vertices. These results are stated as follows.

\begin{theorem}\label{thm-components}
    Let $G$ be any graph and let $G_1,\ldots,G_t$ be the connected components of $G$. Then
    \begin{equation*}
        \ell(J_G) = \sum_{i=1}^t\ell(J_{G_i}).
    \end{equation*}
\end{theorem}

\begin{theorem}\label{thm-add}
    If $G\subset G'$ are graphs on $n$ vertices, then
    \begin{equation*}
        \ell(J_G)\leq \ell(J_{G'}).
    \end{equation*}
\end{theorem}

We work toward the proof of Theorems~\ref{thm-components} and \ref{thm-add} 
by first studying transcendence bases of special fiber rings from a very broad 
perspective.  We first state a porism of the proof of \cite[Proposition 4.8]{HeinzerKim} and a corollary that reduces the computation of
the analytic spread to finding a transcendence base.

\begin{theorem}\label{thm-kit}
    Let $I\subset R$ be a homogeneous ideal generated by forms $F_1,\ldots,F_r$ of the same degree. Then $F(I)$ is isomorphic as rings to the subring $K[F_1,\ldots,F_r]$ of $R$, and in particular, $F(I)$ is a domain.
\end{theorem}
\begin{proof}
    Let  $\mathfrak{m} = (x_1,\ldots,x_n)$ be the graded maximal ideal of $R$. By Definition~\ref{Def: Rees Alg/fib cone}, we have $F(I)\cong \bigoplus_{i\geq 0}I^i/\mathfrak{m}I^i$ as rings. Now the the result immediately follows by the second equation in the proof of \cite[Proposition 4.8]{HeinzerKim}.
\end{proof}

Combining Theorem~\ref{thm-kit} with Remark~\ref{rmk-trdeg} 
gives the following corollary

\begin{corollary}\label{intro-spread}
    Let $I\subset R$ be a homogeneneous
    ideal generated by forms $F_1,\ldots,F_r$ of the same degree. Then.
    \begin{equation*}
        \ell(I) = \trdeg_KK(F_1,\ldots,F_r)
    \end{equation*}which is the cardinality of any transcendence base $\mathcal{B}\subset \{F_1,\ldots,F_r\}$.
\end{corollary}

If $G = (V,E)$ is a finite simple graph, for each
$\{i,j\} \in E$ let $f_{i,j} = x_iy_j -x_jy_i \in S$.  Since
$J_G$ is generated by the forms $f_{i,j}$, which are all
homogeneous of degree two, the previous corollary implies 
computing $\ell(J_G)$ reduces to finding a
transcendence base contained in the set 
$\mathcal{F}:= \{f_{i,j}\mid \{i,j\}\in E\}$; one such procedure
can be found in \cite[Chapter 8, Theorem 1.1]{Lang}.

We will show in Corollary~\ref{corr-3.6} that if $f_1,\ldots,f_r$ are polynomials over $K$ in variables $x_1,\ldots,x_n$ and $g_1,\ldots,g_s$ are polynomials over $K$ in an entirely different set of variables $y_1,\ldots,y_p$, then
\begin{equation}\label{eq: AS break sum}
    \ell((f_1,\ldots,f_r,g_1,\ldots,g_s)) =
    \ell((f_1,\ldots,,f_r))+\ell((g_1,\ldots,g_s)).
\end{equation}
We expect that equation \ref{eq: AS break sum} is well-known but we did not find a proof in the literature, so we included a proof in Corollary~\ref{corr-3.6}. To prove Equation~\ref{eq: AS break sum}, we first establish Lemma~\ref{lem-disjoint}.

\begin{lemma}\label{lem-disjoint}
    Let $n,p,s,r$ be positive integers and let $A = K[x_1,\ldots,x_n,y_1,\ldots,y_p]$ be a polynomial ring over $K$. Let $g_1,\ldots,g_s\in A_1 := K[x_1,\ldots, x_n]$ be algebraically independent forms of the same degree $d$ over $K$, and let $f_1,\ldots,f_r\in A_2:= K[y_1,\ldots,y_p]$ be algebraically independent algebraically independent forms of degree $d$ over $K$. Then the set $\{g_1,\ldots,g_s,f_1,\ldots,f_r\}$ is still algebraically independent over $K$
    \end{lemma}
  
    \begin{proof}
        Suppose that there exists a polynomial $P\in K[X_1,\ldots,X_{r+s}]$ over $K$ such that $P(f_1,\ldots,f_r,g_1,\ldots,g_s)=0$. We must show that $P=0$. Assume for the sake of contradiction that $P\neq 0$. Then there are polynomials $G_{i_1,\ldots,i_s}(X_1,\ldots,X_r)\in K[X_1,\ldots,X_r]$ such that
        \begin{enumerate}
        \item[(i)] $P = \sum_{(i_1,\ldots,i_s)\in\mathbb{N}^s}G_{i_1,\ldots,i_s}(X_1,\ldots,X_r)X_{r+1}^{i_1}\cdots X_{r+s}^{i_s}$; and
            \item[(ii)] there is an index $(j_1,\ldots,j_s)\in\mathbb{N}^s$ such that $G_{j_1,\ldots,j_s}(X_1,\ldots,X_r)\neq 0$.
        \end{enumerate}
        Let $P_0(X_{r+1},\ldots,X_{r+s})\in A_2[X_{r+1},\ldots,X_{r+s}]$ be the polynomial
        \begin{equation*}
            P_0:= P(f_1,\ldots,f_r,X_{r+1},\ldots,X_{r+s})
            = \sum_{(i_1,\ldots,i_s)\in\mathbb{N}^s}G_{i_1,\ldots,i_s}(f_1,\ldots,f_r)X_{r+1}^{i_1}\cdots X_{r+s}^{i_s}.
        \end{equation*}Since $f_1,\ldots,f_r$ are algebraically independent, it follows that $G_{j_1,\ldots,j_s}(f_1,\ldots,f_r)$ is not the zero polynomial in $K[y_1,\ldots,y_p]$. Thus $P_0(y_1,\ldots,y_p,X_{r+1},\ldots,X_{r+s})$ is not the zero polynomial in $K[y_1,\ldots,y_p,X_{r+1},\ldots,X_{r+s}]$. Then since $K$ is infinite there are elements $a_1,\ldots,a_p\in K$ such that
        \begin{equation*}
            0\neq G_{j_1,\ldots,j_s}(f_1(a_1,\ldots,a_p),\ldots,f_r(a_1,\ldots,a_p)).
        \end{equation*}Hence,
        \begin{equation*} 
            0\neq P(f_1(a_1,\ldots,a_p),\ldots,f_r(a_1,\ldots,a_p),X_{r+1},\ldots,X_{r+s}).
        \end{equation*}Let
        \begin{equation*}
           \begin{split}
                P_1(X_{r+1},\ldots,X_{r+s})
            &= P(f_1(a_1,\ldots,a_p),\ldots,f_r(a_1,\ldots,a_p),X_{r+1},\ldots,X_{r+s})
            \\&\in K[X_{r+1},\ldots,X_{r+s}].
           \end{split}
        \end{equation*}
        We have that
        \begin{equation*}
            P_1(g_1,\ldots,g_s)
            =P(f_1(a_1,\ldots,a_p),\ldots,f_r(a_1,\ldots,a_p),g_1,\ldots,g_s)=0.
        \end{equation*}
        which contradicts the algebraic independence of $g_1,\ldots,g_s$.
    \end{proof}

\begin{corollary}\label{corr-3.6}
    Fix positive integers $n,p,s$, and $r$. Let $A = K[x_1,\ldots,x_n,y_1,\ldots,y_p]$ be a polynomial ring over $K$. Take $g_1,\ldots,g_s\in A_1 := K[x_1,\ldots, x_n]$ to be algebraically independent forms of the same degree $d$ and let $f_1,\ldots,f_r\in A_2:= K[y_1,\ldots,y_p]$ be algebraically independent forms of degree $d$. Then
    \begin{equation*}
    \ell((f_1,\ldots,f_r,g_1,\ldots,g_s)) =
    \ell((f_1,\ldots,f_r))+\ell((g_1,\ldots,g_s)).
\end{equation*}
\end{corollary}

\begin{proof}
    By Corollary~\ref{intro-spread} it is enough to show that
    \begin{equation*}
        \trdeg_K K(f_1,\ldots,f_r,g_1,\ldots,g_s)
        =\trdeg_K K(f_1,\ldots,f_r)
        +\trdeg_K K(g_1,\ldots,g_s).
    \end{equation*} By Theorem \ref{thm-transbase}, after reindexing $f_1,\ldots,f_r,g_1,\ldots,g_s$ there are integers $a\leq r$ and $b\leq s$ such that $f_1,\ldots f_a$ is a transcendence base of $K(f_1\ldots,f_r)$ and $g_1,\ldots,g_b$ is a transcendence base of $K(g_1,\ldots,g_s)$. By Lemma~\ref{lem-disjoint}, the set $\mathcal{B} = \{f_1,\ldots,f_a,g_1,\ldots,g_b\}$ is algebraically independent. If $\mathcal{B}$ were not a transcendence base of $K(f_1,\ldots,f_r,g_1,\ldots,g_s)$, then without loss of generality, say $\{f_1,\ldots,f_a,g_1,\ldots,g_b,g_{b+1}\}$ would still be algebraically independent. In particular, $g_1,\ldots,g_b,g_{b+1}$ would be algebraically independent, which contradicts the fact that $g_1,\ldots,g_b$ is a transcendence base of $K(g_1,\ldots,g_s)$. Thus $\{f_1,\ldots,f_a,g_1,\ldots,g_b\}$ is a transcendence base of $K(f_1,\ldots,f_r,g_1,\ldots,g_s)$. This
    fact implies that
    \begin{equation*}
        a+b = \trdeg_K K(f_1,\ldots,f_r,g_1,\ldots,g_s).
    \end{equation*}
    This completes the proof, since $a=\trdeg_K K(f_1,\ldots,f_r)$ and $b = \trdeg_K K(g_1,\ldots,g_s)$.
\end{proof}

We are now ready to prove Theorems~\ref{thm-components} and \ref{thm-add}.
\begin{proof}[Proof of Theorem~\ref{thm-components}.] The assertion follows immediately from Corollary~\ref{corr-3.6} since $J_{G_i}$ and $J_{G_j}$ are
ideals in different sets of variables for all $i\neq j$.
\end{proof}

\begin{proof}[Proof of Theorem~\ref{thm-add}.]
    Let $F_1,\ldots,F_s$ and $F_1,\ldots,F_s,F_{s+1},\ldots,F_r$ be the minimal binomial generators of $J_G$ and $J_{G'}$, respectively.
    By Theorem ~\ref{thm-kit} and Corollary \ref{intro-spread}
    \begin{equation*}
        \ell(J_G)
        =\dim K[J_Gt]=\dim K[F_1,\ldots,F_s]
        =\trdeg_K K(F_1,\ldots,F_s)
    \end{equation*} and similarly $\ell(J_{G'}) = \trdeg_KK(F_1,\ldots,F_r)$. By Theorem~\ref{thm-transbase}, after a suitable reindexing of $F_1,\ldots,F_s$, there exists $p\in\{1,\ldots,s\}$ such that $F_1,\ldots, F_p$ is a transcendence base of $K(F_1,\ldots,F_s)$. Furthermore, by Theorem~\ref{thm-transbase}, there is a set $\mathcal{B}\subset \{F_1,\ldots,F_r\}$ containing $\{F_1,\ldots,F_p\}$ such that $\mathcal{B}$ is a transcendence base of $K(F_1,\ldots,F_r)$ over $K$. Hence
    \begin{equation*}
        \begin{split}
            \ell(J_G) = \trdeg_K K(F_1,\ldots,F_s)=p
            \leq |\mathcal{B}|
            = \trdeg_K K(F_1,\ldots,F_r)=\ell(J_{G'})
        \end{split}
    \end{equation*}which finishes the proof.
\end{proof}

\section{Bounds on the Analytic spread of binomial edge ideals}\label{sec-4}
This section establishes the main formulas on analytic spread of binomial edge ideals of graphs.
We start by computing $\ell(J_G)$ in the 
special case that $G$ is the {\it complete
graph} $K_n$.  Recall that $K_n$ is the graph
with vertex set $[n]$ and edge set $\{\{i,j\}\mid 1\leq i<j\leq n\}$.

\begin{theorem}\label{thm-kn}
    Let $K_n$ denote the complete graph on $n\geq 2$ vertices.  Then $\ell(J_{K_n}) = 2n - 3$.
\end{theorem}
\begin{proof}
    We make use of the observation that $J_{K_n}$
    can be realized as the ideal generated by the $2\times 2$ minors of the matrix $\left [ \begin{matrix}
            x_1 & \ldots & x_n\\
            y_1 & \ldots & y_n
        \end{matrix} \right ].$ This observation allows one to use \cite[Theorem 2]{BrunsSchowanzel} to conclude  $\operatorname{ara}(J_{K_n}) = 2n - 3$. By Remark \ref{rmk-ara} we have $\ell(J_{K_n}) \ge 2n - 3$. Lemma \ref{lem-spread} yields the upper bound
        \begin{equation}\label{eq: anal kn}
            \ell(J_{K_n}) \leq \dim S- \lim_{m \to \infty} \depth (S/(J_{K_n}^m))=2n-\lim_{m \to \infty} \depth (S/(J_{K_n}^m)).
        \end{equation}
But \cite[Theorem 2.1]{Robbiano} shows that $\depth (S/(J_{K_n}^m))=3$ for $m\geq 2$, so the result follows.
\end{proof}

\begin{remark}
 Note that 
    Theorem \ref{thm-kn} gives a nice family
    of examples  to show that the upper 
    bound $\ell(J) \leq \mu(J)$
    is far from optimal.  We have $2n-3 = \ell(J_{K_n}) < 
    \mu(J_{K_n}) = 
    \binom{n+1}{2}$ for $n \gg 0$. In this case $\ell(J_{K_n})$ grows
    linearly with respect on $n$, but $\mu(J_{K_n})$ grows
    quadratically.
\end{remark}

We also need a lower bound 
on the arithmetical rank on the binomial edge ideal.
We say $G$ is {\it $t$-vertex connected} if $t <n$
and $G\setminus S$ is 
connected for all subsets $S \subseteq [n]$
with $|S| < t$.  Here, $G\setminus S$ denotes the graph
$G$ with all vertices and adjacent
edges in $S$ removed.  The {\it vertex connectivity} of $G$ is
the largest integer $t$ such that $G$ is $t$-vertex 
connected.

\begin{theorem}[{\cite[Theorem 3.5]{kats}}]\label{arabound}
Suppose $G$ is a connected graph on $n$ vertices with
vertex connectivity $t$.
Then 
  $\operatorname{ara}(J_G) \geq n+t-2$.
\end{theorem}
  
We can now prove our first main result of this section.

\begin{theorem}\label{bounds}
If $G$ is a connected graph on $n\geq 2$ vertices, then $n - 1 \leq \ell(J_G) \leq 2n - 3$.    
\end{theorem}

\begin{proof}
From Theorem \ref{arabound} we get $\ara(J_G)\geq n+t-2$ where $t$ is the vertex connectivity of $G$. The connectivity assumption on $G$ implies $t\geq 1$, and 
thus by Remark \ref{rmk-ara} we have
\begin{equation*}
    n-1\leq \ara(J_G) \leq \ell(J_G).
\end{equation*}
By Lemma~\ref{thm-add}, $\ell(J_G) \le \ell(J_{G'})$, where $G'$ is formed from $G$ by adding  a new edge
between vertices that are currently not joined.  By
repeatedly adding edges in this way, the graph $K_n$
is eventually created. Thus $\ell(J_G) \leq \ell(J_{K_n})$. Combining this fact with Theorem~\ref{thm-kn} we have $\ell(J_G) \leq 2n - 3$.
\end{proof}

\begin{remark}
    Theorem \ref{thm-kn} shows that the upper bound is
    tight.  In Section \ref{sec-exact} we will
    show that if $G$ is a tree on $n$ vertices, then
    $\ell(J_G) = |E| = n-1$.  So both bounds are tight.
\end{remark}

As a consequence of Theorem \ref{bounds} and some elementary graph theory, we obtain an additional upper bound for the analytic spread of a binomial edge ideal. 
Recall that for $v\in V$, the number of edges adjacent to $v$ is the \textit{degree of $v$}, denoted $\deg(v)$.
The {\it minimum degree} of $G$ is the integer
$\min\{\deg(v)\mid v\in V\}.$  

\begin{corollary}\label{cor: anals < E}
Suppose $G$ is a connected graph whose minimal
degree is at least 4.
Then
 \begin{equation*}
     \ell(J_G)\leq |E|-3.
 \end{equation*}
\end{corollary}
\begin{proof}
Suppose $G$ has $n$ vertices.  By the Handshaking
Theorem, $4n \leq \sum_{v \in V} \deg(v) = 2|E|$, and
thus  $2n\leq |E|$. Then by Theorem \ref{bounds}
we have $
        \ell(J_G)\leq 2n-3 \leq |E|-3.$
\end{proof}


\section{Exact values of the analytic spread}\label{sec-exact}
In this section we compute the precise value of the analytic spread of the binomial edge ideals of unicyclic and pseudo-forest graphs.
In these cases,
$\ell(J_G) = \mu(J_G) = |E|$, thus giving
examples where the upper bound of $\mu(J_G)$ on $\ell(J_G)$ is achieved. Recall that a {\it pseudo-tree} is a connected graph with at most one cycle, and a {\it unicyclic graph} $G$ is a connected graph with exactly one cycle.

Now we work towards proving Corollary~\ref{corr-tree}, which is used in the proof of Theorem~\ref{thm-main}. To state Proposition \ref{corr-leaf}, from which Corollary \ref{corr-tree} follows, we need the following definition.
\begin{definition}
    Let $G = (V,E)$ be a graph and let $G' = (V',E')$ be another graph such that $V\subset V'$, $|V'| = |V|+1$, and $E' = E\cup \{e\}$, where $e = \{i,j\}$, $i\in V$ and $\{j\} = V'\setminus V$. Then we say that $G'$ is a graph obtained by \textit{adding a leaf to $G$}.
\end{definition}
In Proposition~\ref{corr-leaf}, we investigate the behavior of analytic spread under the operation of adding a leaf. This allows us to compute the analytic spread of the binomial edge ideal of a tree in Corollary~\ref{corr-tree}.
\begin{proposition}\label{corr-leaf}
    Let $G$ be a graph on $n-1\geq 1$ vertices. Let $G'$ be a graph on $n$ vertices obtained by adding a leaf to $G$. Then $\ell(J_{G'}) = \ell(J_G)+1$.
\end{proposition}

\begin{proof}
Recall that $J_G$ is an ideal in a polynomial ring with $2(n-1)$ variables, that is,
\begin{equation*}
    J_G \subset K[x_1,\ldots,x_{n-1},y_1,\ldots,y_{n-1}]\subset K[x_1,\ldots,x_n,y_1,\ldots,y_n]=S.
\end{equation*}
Let $F_1,\ldots,F_l\in S$ be the minimal binomial generators of $J_G$ and let $F\in S$ be the binomial corresponding to the given leaf, so that $F = x_iy_n-x_ny_i$ for some $i\in\{1,\ldots,n-1\}$. By Theorem~\ref{thm-transbase}, there is a subset $F_1,\ldots,F_s$ of $\{F_1,\ldots,F_l\}$ such that $F_1,\ldots,F_s$ is a transcendence base of $K(F_1,\ldots,F_l)$. In particular, $\trdeg_K F(J_G) = s$ by Lemma~\ref{intro-spread}. Let $A = K[F_1t,\ldots,F_st,Ft]$. By Theorem~\ref{thm-transbase}, $F_1t,\ldots,F_st,Ft$ contains a transcendence base of $K(F_1t,\ldots,F_st,Ft)$.

Note that $K[F_1t,\ldots,F_st]$ and $K[F_1,\ldots,F_s]$ have the same dimension (as there are ring surjections both ways given by $F_it\mapsto F_i$ and $F_i\mapsto F_it$). Thus $F_1t,\ldots,F_st$ is still algebraically independent over $K$.

Next, we prove that $F_1t,\ldots,F_st,Ft$ is a transcendence base of $K(F_1t,\ldots,F_st,Ft)$. Consider the ring map 
\begin{equation*}
    \phi: A =K[F_1t,\ldots,F_st,Ft] 
    \to K[x_1\ldots,x_n,y_1,\ldots,y_{n},t]
\end{equation*} given by $F_it\mapsto F_it$ for $1\leq i\leq s$ and $Ft\mapsto x_n(x_i-y_i)t$. Let $L = \phi(Ft) = x_n(x_i-y_i)t$ and note that $F_1t,\ldots,F_st,L$ contains a transcendence base of $\text{QF}(B)$, where $B:=\im(\phi)$.
Suppose that there is a polynomial $P:=P(X_1,\ldots,X_{s+1})$ over $K$ such that $P(F_1t,\ldots,F_st,L) = 0$. Substituting $t=1$ into this equation we have that $P(F_1,\ldots,F_s,x_n(x_i-y_i))=0$. It suffices to show that $P=0$. One can rewrite $P$ as
\begin{equation}\label{eq2}
\end{equation}
where $r\geq 0$ and $Q_0,\ldots,Q_r$ are polynomials in $K[X_1,\ldots,X_s]$. Substituting $F_1,\ldots,F_s$, $x_n(x_i-y_i)$ for $X_1,\ldots,X_{s+1}$ respectively into (\ref{eq2}), we arrive at the equation
\begin{equation*}
    0 = Q_0(F_1,\ldots,F_s)+ Q_1(F_1,\ldots,F_s)x_n(x_i-y_i)+\cdots + Q_r(F_1,\ldots,F_s)x_n^r(x_i-y_i)^r.
\end{equation*}Considering the preceding equation as an equality of polynomials in the variable $x_n$ yields that $Q_j(F_1,\ldots,F_s)(x_i-y_i)^j=0$, and since $(x_i-y_i)^j\neq 0$, it follows that $Q_j(F_1,\ldots,F_s)=0$ for all $0\leq j\leq r$. Since $F_1,\ldots,F_s$ are algebraically independent over $K$, each polynomial $Q_j$ must be the zero polynomial. Thus $Q_0=\cdots=Q_r=0$ and subsequently $P=0$. Therefore $F_1t,\ldots,F_st,L$ are algebraically independent over $K$, which yields the inequality
\begin{equation*}
    \dim(A) \geq \dim(B) = \trdeg_K {\rm QF}(B) = s+1.
\end{equation*}On the other hand, since $A=K[F_1t,\ldots,F_st,Ft]$ and $F_1t,\ldots,F_st,Ft$ contains a transcendence base of $\text{QF}(A)$, $\dim(A) = \trdeg_K \text{QF}(A)\leq s+1$. So $\trdeg K(F_1t,\ldots,F_st,Ft)=s+1$. Thus, $K(F_1t,\ldots, F_lt,Ft)$ has transcendence degree $s+1 = \trdeg_KK(F_1t,\ldots,F_lt) + 1$. Finally, Remark~\ref{rmk-trdeg} implies the following equation which finishes the proof.
\begin{equation*}
    \begin{split}
        \ell(J_{G'}) &= \dim(K[J_{G'}t]) = \trdeg_KK(F_1t,\ldots,F_lt,Ft)
        \\&= \trdeg_KK(F_1t,\ldots,F_lt)+1=\dim(K[J_Gt])+1 = \ell(J_{G})+1. \qedhere
    \end{split}
\end{equation*} 
\end{proof}

\begin{corollary}\label{corr-tree}
    Let $T$ be a tree. Then $\ell(T) = |E|$.
\end{corollary}
\begin{proof}
    Firstly, if $T$ is a single edge $\{1,2\}$, then Lemma~\ref{intro-spread} implies
    \begin{equation*}
        \ell(T) = \dim K[x_1y_2-x_2y_1]=1.
    \end{equation*}The result now follows by Proposition~\ref{corr-leaf}.
\end{proof}

Similar to adding a leaf to a graph and observing how the analytic spread of $J_G$ changes, we can observe how $\ell(J_G)$ changes under the operation of adding a handle to the graph; this means adding a vertex labeled $n+1$ to a graph $G$ on $n$ vertices and edges $\{i,n+1 \}$ and $\{j,n+1 \}$ where $i\neq j$ and $1\leq i,j \leq n$.

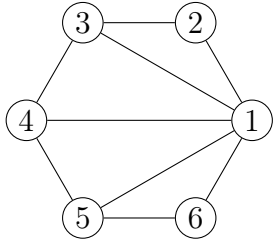
\begin{example}\label{ex: handle}
\
\medskip

\begin{minipage}{.5\textwidth}\centering
\begin{tikzpicture}[scale=1.5, every node/.style={circle, draw, fill=white, inner sep=2pt}]
  \node (v1) at (0:1) {1};
  \node (v2) at (60:1) {2};
  \node (v3) at (120:1) {3};
  \node (v4) at (180:1) {4};
  \node (v5) at (240:1) {5};
  \node (v6) at (300:1) {6};
  \draw (v1) -- (v2);
  \draw (v2) -- (v3);
  \draw (v3) -- (v4);
  \draw (v4) -- (v5);
  \draw (v5) -- (v6);
  \draw (v6) -- (v1);
  \foreach \i in {v3,v4,v5}
  \draw (v1)--(\i);
\end{tikzpicture}
\end{minipage}
\hspace{-1 cm}
\begin{minipage}{.5\textwidth}
 Let $G$ be the graph on  $V=\{1,\ldots,6\}$ that is pictured.  Note that $\{5,6\}$ and $\{1,6\}$ form
 a handle for $G'$, the induced graph on $\{1,\ldots,5\}$.
 Using Macaulay2 one can show that $\ell(J_{G'}) =7$
 and $\ell(J_G) = 9$.
 \end{minipage}
\end{example}
Example~\ref{ex: handle} and similar computations in Macaulay2 lead to the following question.
\begin{question}\label{Quesiton: handle}
    Let $G$ be a planar graph on $n\geq 3$ vertices. If $G'$ is a planar graph obtained from $G$ by adding a handle, then is it always true that
    $
        \ell(J_{G'})=\ell(J_{G})+2?
    $
 \end{question}

The last step in the proof of Theorem \ref{thm-main} is the proof of Theorem \ref{thm-uni}, which follows from Lemma \ref{lem-cycle}.

\begin{lemma}\label{lem-cycle}
Let $C_n$ be the $n$-cycle.  Then $\ell(J_{C_n}) = n$. 
\end{lemma}

\begin{proof} By \cite[Theorem 3.7]{kats}, $\operatorname{ara}(J_{C_n}) = n$.  Since $\mu(J_{C_n}) = n$ and $\operatorname{ara}(J_{C_n}) \le \ell(J_{C_n}) \le \mu(J_{C_n})$, we are done.
\end{proof}
\begin{theorem}\label{thm-uni}
    Let $G=(V,E)$ be a connected unicyclic graph. Then
    $
        \ell(J_G) = |E|.
    $
\end{theorem}
\begin{proof}
Let $r = |E|$. Let $i$ be the length of the cycle $C_i$ in $G$. The graph $C_i$ has $i$ edges.  The graph $G$ is then
formed by adding $r-i$ edges, one at a time.  Furthermore, we can build $G$ from $C_i$ by insuring that the edge we add each step is a leaf.
We have $\ell(J_{C_i})=i$ by Lemma~\ref{lem-cycle}. Then by the Proposition~\ref{corr-leaf}
\begin{equation*}
    \ell(J_G) = \ell(J_{C_i}) + l
    = i + (r-i) = r.
\end{equation*}This completes the proof of the theorem.
\end{proof}
\begin{definition}
    A graph $G$ is called a pseudo-forest if every connected component of $G$ has at most one cycle.
\end{definition}

Now we recall and prove Theorem~\ref{thm-main}.

\begin{theorem}\label{thm-main}
    Let $G$ be a pseudo-forest. Then $\ell(J_G)$ is the number of edges in $G$. In particular, if $G$ is a forest or a unicyclic graph, then the $\ell(J_G)$ is the number of edges in $G$.
\end{theorem}
\begin{proof}
    The result follows from Corollary~\ref{corr-tree}, Theorem~\ref{thm-uni} and Lemma~\ref{thm-components}.
\end{proof}


\section{Analytic Spread of Closed Graphs and Newton-Okounkov Regions}\label{sec-6}

In this section, we use convex regions constructed from the binomial edge ideal of a closed graph to understand its analytic spread. We shall consider the Newton-Okounkov region associated to a graded family of monomial ideals that was constructed in \cite{HN1, KK2}. This construction is closely related to that of the Newton-Okounkov body, which was investigated in \cite{KK}. However, Newton-Okounkov regions are generally not compact, as opposed to Newton-Okounkov bodies. We also classify all closed graphs $G$ for which the analytic spread of $J_G$ is equal to the number of edges of $G$.

Recall that a \emph{graded family} of ideals in $S$ is a collection $\cI = \{I_i\}_{i \ge 0}$ of ideals such that $I_0 = R$ and $I_p \cdot I_q \subseteq I_{p+q}$ for all $p,q \ge 1$. For $a = (a_1, \dots, a_{2n}) \in \Z_{\ge 0}^{2n}$, we shall use $X^a$ to denote the monomial $x_1^{a_1} x_2^{a_2} \cdots y_{n-1}^{a_{2n-1}}y_n^{a_{2n}}$ in $S$. If $\cI = \{I_i\}_{i \ge 0}$ is a graded family of ideals in $S$, then we define the \emph{Rees algebra of $\cI$} to be $R[\mathcal{I}] = \bigoplus_{i\geq 0}I_it^i$. The \emph{special fiber ring} of $\mathcal{I}$ is $F(\mathcal{I}):= R[\mathcal{I}]/mR[\mathcal{I}]$. The \emph{analytic spread} $\ell(\mathcal{I})$ is $\dim F(\mathcal{I})$; see \cite{spread1} and \cite{spread2} for a discussion of analytic spread of graded families.
 
\begin{definition} Let $\cI = \{I_i\}_{i \ge 0}$ be a graded family of monomial ideals in $S$. The \emph{Newton-Okounkov} region of $\cI$ is defined as
        $$\Delta(\cI) = \overline{\bigcup_{k \ge 1} \text{convex hull} \left\{ \dfrac{a}{k} ~\Big|~ X^a \in I_k\right\}}\subset \mathbb{R}^{2n}.$$
Moreover, let $\mathrm{mcd}(\Delta)$ be the {\it maximum dimension of a compact face of} a polyhedron $\Delta$.
\end{definition}

Consider the lexicographical monomial ordering on $S = K[x_1, \ldots, x_n, y_1, \ldots, y_n]$ induced by the prescription that $x_1 > x_2 > \dots > y_1 > \dots > y_n$. For an ideal $I \subseteq S$, we shall use $\text{in}_<(I)$ to denote the initial ideal of $I$ under this monomial ordering.

\begin{definition}
    Let $G$ be a graph and let $J_G$ be its binomial edge ideal. Set $\cI = \{\text{in}_<(J_G^i)\}_{i \ge 0}$. We define
    \begin{equation}\label{eq: NO body}
        \Delta(J_G) = \Delta(\cI).
    \end{equation}
\end{definition}

\begin{remark}
    In general, Newton-Okounkov regions are constructed with respect to a \emph{good} valuation (cf. \cite{KK2}). It is not hard to see that the Newton-Okounkov region of the graded family $\{J_G^i\}_{i \ge 0}$, constructed from the Gr\"obner valuation of $S$ with respect to the reverse lex ordering, where $y_n > \dots > y_1 > x_n > \dots > x_1$, coincides with $\Delta(J_G)$.
\end{remark}

Recall from \cite[Theorem 1.1]{HHHKR} that a graph $G$ is \emph{closed} if the binomial edge ideal $J_G$ has a quadratic Gr\"obner basis with respect to a diagonal term order. This class of graphs has received considerable attention; see, for instance, \cite{ene2022powers, PeevaClosed}. We are now ready to prove our next result giving a precise formula for the analytic spread of the binomial edge ideal of closed graphs. 

\begin{theorem}\label{thm-closed1}
    Let $G$ be a closed graph. Then $\ell(J_G) = \mathrm{mcd}(\Delta(J_G))+ 1$.
\end{theorem}

\begin{proof}
    Since $G$ is closed, ${\rm in}_{<}(J_G^k) = 
    ({\rm in}_{<}(J_G))^k$ by \cite[Equation (3)]{ene2022powers}. Then from Equation~ \ref{eq: NO body} we get, $\Delta(J_G) = \Delta({\rm in}_{<} (J_G))$. Since $G$ is closed, by \cite[Theorem 3.10]{ene2022powers}, $\ell(J_G) = 
    \ell({\rm in}_{<}(J_G))$. Note that if $\mathcal{I}$ is a graded family of ordinary powers of an ideal $I$, then $\ell(\mathcal{I}) = \ell(I)$. Therefore by \cite[Theorem 4.1]{HN1}, $\ell(J_G) = 
    \ell({\rm in}_{<}(J_G)) = 
    \mathrm{mcd}(\Delta({\rm in}_{<}(J_G)))+ 1 = \mathrm{mcd}(\Delta(J_G))+1$. 
\end{proof}

\begin{corollary}\label{cor: mcd}
    Let $\mathcal{I}$ be the family of symbolic powers of $J_G$. If $G$ is closed, then $\ell(\mathcal{I}) = \mathrm{mcd}(\Delta(J_G))+ 1$.
\end{corollary}

\begin{proof}
    Since $G$ is closed, by \cite[Equation (4)]{ene2022powers}, we have $J_G^k = J_G^{(k)}$, where $J_G^{(k)}$ is the $k$th symbolic power of $J_G$. Hence the proof follows.
\end{proof}

The paper ends with Corollary~\ref{cor: Pluck and bound}, in which we recover a result of Kumar \cite{kumar2022rees} in fully characterizing $K_4$-free closed graphs based on the analytic spread of their binomial edge ideals. We say that a graph $G$ is $K_4$-free if $G$ does not contain $K_4$ as an induced subgraph. We introduce enough definitions and notation to recall Theorem~\ref{thm: rees} from \cite{ALMOUSA2024}. 
\begin{definition}\label{def: All Reese}
    Let $G=(V,E)$ be any graph on $n\geq 4$ vertices. One can define the following ring and ideal associated to $G$. Recall that $S=K[x_1,\ldots,x_n,y_1,\ldots,y_n]$.
    Set
    \begin{equation*}
    \begin{split}        
        S_G &= S[T_{ij}\mid \{i,j \} \in E(G)],\\
        I_{G} &= (T_{i_1i_2}T_{i_3i_4}+T_{i_2i_3}T_{i_1i_4}-T_{i_1i_3}T_{i_2i_4} \mid  \{i_1,i_2,i_3,i_4\} \subseteq V \text{ induce a $K_4$ subgraph of $G$} ) \subseteq S_G,\\
        F_G&=\frac{K[T_{ij}\mid \{i,j \} \in E(G)]}{I_G \cap K[T_{ij}\mid \{i,j \}\in E(G)]}.
    \end{split}
    \end{equation*} 
    Define $\varphi$ as the surjective map
    \begin{equation}\label{def: Reese Map}     
        \varphi:S_G \twoheadrightarrow S[J_Gt]=\bigoplus_{n\geq 0}J_G^nt^n,
    \end{equation} 
where $\varphi(T_{i,j})=f_{i,j}t$ and $\varphi(s)=s$ for all $s$ in $S$. The kernel of $\varphi$ is called the {\it presentation ideal of $J_G$}.
\end{definition}

Note that {\cite[Theorem 4.5]{ALMOUSA2024}} is stated for a more general class of ideals that includes binomial edge ideals, see \cite[Definition 2.1]{ALMOUSA2024}. We have tailored the statement to our setting. Recall that $F(J_G)=S[J_Gt]/\mathfrak{m}S[J_Gt]$.
\begin{theorem}[{\cite[Theorem 4.5]{ALMOUSA2024}}]\label{thm: rees}
Let $G$ be a closed graph with maximal cliques $\Delta_1, \ldots, \Delta_s$ that cover all vertices of $G$. For $\{i,j\}\in E(G)$ write $\{i,j\}\in \Delta_a$ if $a$ is the smallest index of a clique that contains $\{i,j\}$. The presentation ideal of $J_G$ is generated by:
\begin{enumerate}
\item $f_{ij}T_{i'j'}-f_{i'j'}T_{ij}$ for  $\{i,j\}\in\Delta_a, \{i',j'\}\in \Delta_b$ for $a\neq b$  {\em (Koszul relations)},
\item  $x_iT_{jk}-x_jT_{ik}+x_kT_{ij}$  and $y_iT_{jk}-y_jT_{ik}+y_kT_{ij}$ for $\{i,j,k\}$  an induced $C_3$ in $G$ such that $i<j<k$ {\em (Eagon-Northcott relations)},
\item $T_{ij}T_{kl} -T_{ik}T_{jl} + T_{il}T_{jk}$ where $1\leq i<j<k<l\leq n$ and $\{i, j, k, l\}$  induce a $K_4$ in $G$ 
{\em (Pl\"ucker relations)}.
\end{enumerate}
In particular, as $K$-algebras,
 \begin{equation}\label{eq: fib cone to Pluck}
F_G\cong F(J_G). 
    \end{equation}
\end{theorem}
\begin{corollary}\label{cor: Pluck and bound}
    Let $G=(V,E)$ be a closed graph on $n$ vertices.
    \begin{enumerate}
        \item $\ell(J_G)\leq|E|$, and $\ell(J_G)=|E|$ if and only if $G$ contains no induced $K_4$.
        \item The dimension of $ F_G$ is equal to the following quantities:
        \begin{itemize}
            \item the number of algebraically independent elements of the set 
            \begin{equation*}
                \{ \overline{T_{ij}} \mid (i,j)=(i,i+1)\ \text{or}\ i=1\} \subseteq F_G,
            \end{equation*}
            \item $\mathrm{mcd}(\Delta(J_G))+1$.
        \end{itemize}
    \end{enumerate}
\end{corollary}
\begin{proof}
$(1)$ Since $S_G$ is a polynomial ring, $G$ contains an induced $K_4$ if and only if $1\leq \text{ht}(I_G)$ and the result follows by Equation~\ref{eq: fib cone to Pluck}.\\
$(2)$ Consider the set $T=K[\overline{T_{ij}}\ :\ (i,j)=(i,i+1)\ \text{or}\ i=1 \text{ for } 1\leq i \leq n]\subseteq F_G$. Since $T_{12}T_{34}+T_{14}T_{23}-T_{13}T_{24}\in I_G$ then $\overline{T_{24}}\in \text{QF}(T)$. Similarly, since $T_{12}T_{45}+T_{15}T_{24}-T_{14}T_{25}\in I_G$ and $\overline{T_{24}}\in\text{QF}(T)$, then $\overline{T_{25}}\in\text{QF}(T)$. Continuing in this fashion, we get $\overline{T_{2i}}\in\text{QF}(T)$ for $3\leq i\leq n$. Replicating the same arguments we can get that $\overline{T_{ij}}\in\text{QF}(T)$ for any $1\leq i<j\leq n$ and thus
\begin{equation*}
   \dim T= \dim F_G \underset{(*)}{=}\dim(F(J_G))=\ell(J_G)\underset{(**)}{=}\mathrm{mcd}(\Delta(J_G))+1,
\end{equation*}
where $(*)$ and $(**)$ follow from Equation~\ref{eq: fib cone to Pluck} and Theorem~\ref{thm-closed1} respectively.
\end{proof}

\subsection*{Acknowledgments}

Work on this project started 
at the Fields Institute in Toronto, Canada
as part of the ``Apprenticeship Program in Commutative Algebra''.  All of the authors thank
the organizers and Fields for providing an wonderful environment in which to work and for funding to travel
to Toronto.  We would like to thank Sudipta Das for his 
feedback on eariler versions of this project. Dey was partly supported by the Charles University Research Center program No.UNCE/24/SCI/022 and a grant GA \v{C}R 23-05148S from the Czech Science Foundation. H\`a acknowledges support from the Simons Foundation.
Landsittel acknowledges his support from the Center for Mathematical Sciences and Applications at Harvard.
Roshan Zamir received support from the NSF grant DMS-2342256 RTG: Commutative Algebra at Nebraska. 
Van Tuyl’s research is supported by NSERC Discovery Grant 2024-05299.

\bibliographystyle{amsplain}

\bibliography{biblio}

\providecommand{\bysame}{\leavevmode\hbox to3em{\hrulefill}\thinspace}
\providecommand{\MR}{\relax\ifhmode\unskip\space\fi MR }
\providecommand{\MRhref}[2]{%
  \href{http://www.ams.org/mathscinet-getitem?mr=#1}{#2}
}
\providecommand{\href}[2]{#2}
\begin{thebibliography}{10}

\bibitem{ALMOUSA2024}
Ayah Almousa, Kuei-Nuan Lin, and Whitney Liske, \emph{Rees algebras of unit
  interval determinantal facet ideals}, J. Pure Appl. Algebra \textbf{228}
  (2024), no.~5, Paper No. 107601, 15. \MR{4688113}

\bibitem{bivia2003analytic}
Carles Bivi\`a-Ausina, \emph{The analytic spread of monomial ideals}, Comm.
  Algebra \textbf{31} (2003), no.~7, 3487--3496. \MR{1990285}

\bibitem{bolognini2022cohen}
Davide Bolognini, Antonio Macchia, Giancarlo Rinaldo, and Francesco Strazzanti,
  \emph{Cohen-{M}acaulay binomial edge ideals of small graphs}, J. Algebra
  \textbf{638} (2024), 189--213. \MR{4655685}

\bibitem{BrunsSchowanzel}
Winfried Bruns and Roland Schw\"anzl, \emph{The number of equations defining a
  determinantal variety}, Bull. London Math. Soc. \textbf{22} (1990), no.~5,
  439--445. \MR{1082012}

\bibitem{Burch_1972}
Lindsay Burch, \emph{Codimension and analytic spread}, Proc. Cambridge Philos.
  Soc. \textbf{72} (1972), 369--373. \MR{304377}

\bibitem{PoliniUlrichCarsoGhezzi}
Alberto Corso, Laura Ghezzi, Claudia Polini, and Bernd Ulrich,
  \emph{Cohen-{M}acaulayness of special fiber rings}, vol.~31, Taylor \&
  Francis, 2003, Special issue in honor of Steven L. Kleiman, pp.~3713--3734.
  \MR{2007381}

\bibitem{Cut1}
Steven~Dale Cutkosky, \emph{Asymptotic multiplicities of graded families of
  ideals and linear series}, Adv. Math. \textbf{264} (2014), 55--113.
  \MR{3250280}

\bibitem{cutkosky2024reesalgebrasreducedfiber}
\bysame, \emph{Rees algebras and the reduced fiber cone of divisorial
  filtrations on two dimensional normal local rings}, Comm. Algebra \textbf{53}
  (2025), no.~8, 3364--3387. \MR{4911127}

\bibitem{CL}
Steven~Dale Cutkosky and Stephen Landsittel, \emph{Epsilon multiplicity is a
  limit of amao multiplicities}, Journal of Algebra and its Applications
  (2024).

\bibitem{spread1}
Steven~Dale Cutkosky and Parangama Sarkar, \emph{Analytic spread of filtrations
  and symbolic algebras}, J. Lond. Math. Soc. (2) \textbf{106} (2022), no.~3,
  2635--2662. \MR{4498562}

\bibitem{spread2}
\bysame, \emph{Epsilon multiplicity and analytic spread of filtrations},
  Illinois J. Math. \textbf{68} (2024), no.~1, 189--210. \MR{4720561}

\bibitem{ene2022powers}
Viviana Ene, Giancarlo Rinaldo, and Naoki Terai, \emph{Powers of binomial edge
  ideals with quadratic {G}r\"obner bases}, Nagoya Math. J. \textbf{246}
  (2022), 233--255. \MR{4425287}

\bibitem{HN1}
Huy~T\`ai H\`a and Th\'ai~Th\`anh Nguyen, \emph{Newton-{O}kounkov body, {R}ees
  algebra, and analytic spread of graded families of monomial ideals}, Trans.
  Amer. Math. Soc. Ser. B \textbf{11} (2024), 1065--1097. \MR{4806407}

\bibitem{HeinzerKim}
William~J. Heinzer and Mee-Kyoung Kim, \emph{Properties of the fiber cone of
  ideals in local rings}, Comm. Algebra \textbf{31} (2003), no.~7, 3529--3546.
  \MR{1990289}

\bibitem{MR2724673}
J\"urgen Herzog and Takayuki Hibi, \emph{Monomial ideals}, Graduate Texts in
  Mathematics, vol. 260, Springer-Verlag London, Ltd., London, 2011.

\bibitem{HHHKR}
J\"urgen Herzog, Takayuki Hibi, Freyja Hreinsd\'ottir, Thomas Kahle, and
  Johannes Rauh, \emph{Binomial edge ideals and conditional independence
  statements}, Adv. in Appl. Math. \textbf{45} (2010), no.~3, 317--333.
  \MR{2669070}

\bibitem{herzog2019fiberconemonomialideals}
J\"urgen Herzog and Guangjun Zhu, \emph{On the fiber cone of monomial ideals},
  Arch. Math. (Basel) \textbf{113} (2019), no.~5, 469--481. \MR{4016631}

\bibitem{SH}
Craig Huneke and Irena Swanson, \emph{Integral closure of ideals, rings, and
  modules}, London Mathematical Society Lecture Note Series, vol. 336,
  Cambridge University Press, Cambridge, 2006. \MR{2266432}

\bibitem{kats}
Anargyros Katsabekis, \emph{Arithmetical rank of binomial ideals}, Arch. Math.
  (Basel) \textbf{109} (2017), no.~4, 323--334. \MR{3698319}

\bibitem{KK}
Kiumars Kaveh and Askold Khovanskii, \emph{Newton-{O}kounkov bodies, semigroups
  of integral points, graded algebras and intersection theory}, Ann. of Math.
  (2) \textbf{176} (2012), no.~2, 925--978. \MR{2950767}

\bibitem{KK2}
\bysame, \emph{Convex bodies and multiplicities of ideals}, Proc. Steklov Inst.
  Math. \textbf{286} (2014), no.~1, 268--284, Reprint of Tr. Mat. Inst.
  Steklova {\bf 286} (2014), 291--307. \MR{3482603}

\bibitem{kumar2022rees}
Arvind Kumar, \emph{Rees algebra and special fiber ring of binomial edge ideals
  of closed graphs}, Illinois J. Math. \textbf{66} (2022), no.~1, 79--90.
  \MR{4405525}

\bibitem{Lang}
Serge Lang, \emph{Algebra}, third ed., Graduate Texts in Mathematics, vol. 211,
  Springer-Verlag, New York, 2002. \MR{1878556}

\bibitem{LM}
Robert Lazarsfeld and Mircea Musta\c{t}\v{a}, \emph{Convex bodies associated to
  linear series}, Ann. Sci. \'Ec. Norm. Sup\'er. (4) \textbf{42} (2009), no.~5,
  783--835. \MR{2571958}

\bibitem{lin2021}
Kuei-Nuan Lin and Yi-Huang Shen, \emph{Fiber cones of rational normal scrolls
  are {C}ohen-{M}acaulay}, J. Algebraic Combin. \textbf{56} (2022), no.~2,
  547--563. \MR{4467966}

\bibitem{ReeseNortcotOG}
D.~G. Northcott and D.~Rees, \emph{Reductions of ideals in local rings}, Proc.
  Cambridge Philos. Soc. \textbf{50} (1954), 145--158. \MR{59889}

\bibitem{ohtani2011graphs}
Masahiro Ohtani, \emph{Graphs and ideals generated by some 2-minors}, Comm.
  Algebra \textbf{39} (2011), no.~3, 905--917. \MR{2782571}

\bibitem{PeevaClosed}
Irena Peeva, \emph{Closed binomial edge ideals}, J. Reine Angew. Math.
  \textbf{803} (2023), 1--33. \MR{4649181}

\bibitem{Robbiano}
Lorenzo Robbiano, \emph{An algebraic property of $\mathbb{P}^1 \times
  \mathbb{P}^n$}, Comm. Algebra \textbf{7} (1979), no.~6, 641--655.

\bibitem{saloni2022}
Kumari Saloni, \emph{Multiplicity versus {B}uchsbaumness of the special fiber
  cone}, 2022.

\bibitem{WT}
Hong Wang and Zhongming Tang, \emph{Depth of powers of binomial edge ideals of
  complete bipartite graphs}, Comm. Algebra \textbf{51} (2023), no.~6,
  2472--2483. \MR{4563443}

\end{thebibliography}

\end{document}